\documentclass[10pt,final]{amsart}
\usepackage{latexsym,amssymb}
\usepackage{rotating}
\usepackage{float}

\usepackage{epic,eepic}

\newtheorem{theorem}{Theorem}

\newtheorem{lemma}[theorem]{Lemma}
\newtheorem{conjecture}{Conjecture}
\renewcommand{\theconjecture}{\Alph{conjecture}}

\usepackage[T1]{fontenc}
\usepackage[utf8]{inputenc}
\usepackage{lmodern}
\usepackage{amsmath}
\usepackage{xcolor}
\usepackage{enumitem}
\definecolor{linkB}{rgb}{0.00 0.2 0.48}
\usepackage[colorlinks, linkcolor=linkB, citecolor=linkB, urlcolor=linkB]{hyperref}
\newcommand{\lemmaOne}{\hyperref[lemma:1]{Lemma 1}}
\newcommand{\lemmaTwo}{\hyperref[remark:1]{Remark 1}}

\title{Hoffmann-Ostenhof's conjecture for traceable cubic graphs}

\newcommand*{\affaddr}[1]{\centering{\small \em #1}} 
\newcommand*{\affmark}[1][*]{\textsuperscript{#1}}

\thanks{s\_akbari@sharif.edu, \{abdolhosseini, hohashemi, sadramoradian\}@ce.sharif.edu}

\begin{document}

\maketitle

\vspace{-4mm}
{\centering F. Abdolhosseini\affmark[a], S. Akbari\affmark[b], H. Hashemi\affmark[a], M.S. Moradian\affmark[a]\\
}
\vspace{6mm}
{\affaddr{\affmark[a]Department of Computer Engineering, Sharif University of Technology, Tehran, Iran}\\
}
{\affaddr{\affmark[b]Department of Mathematical Sciences, Sharif University of Technology, Tehran, Iran}\\
}


\begin{abstract}
It was conjectured by Hoffmann-Ostenhof that the edge set of every connected cubic graph can be
decomposed into a spanning tree, a matching and a family of cycles. In this paper, we show that this conjecture holds for traceable cubic graphs.\\

\noindent \textsc{keywords:} Cubic graph, Hoffmann-Ostenhof's Conjecture, Traceable
\end{abstract}



\section{Introduction}
Let $G$ be a simple undirected graph with the \textit{vertex set} $V(G)$ and the \textit{edge set} $E(G)$. A vertex with degree one is called a \textit{pendant vertex}. The distance between the vertices $u$ and $v$ in graph $G$ is denoted by $d_G(u,v)$. A cycle $C$ is called \textit{chordless} if $C$ has no \textit{cycle chord} (that is an edge not in the edge set of $C$ whose endpoints lie on the vertices of $C$).
The \textit{Induced subgraph} on vertex set $S$ is denoted by $\langle S\rangle$. A path that starts in $v$ and ends in $u$ is denoted by $\stackrel\frown{v u}$.
A \textit{traceable} graph is a graph that possesses a Hamiltonian path.
In a graph $G$, we say that a cycle $C$ is \textit{formed by the path} $Q$ if $ | E(C) \setminus E(Q) | = 1 $. So every vertex of $C$ belongs to $V(Q)$.

In 2011 the following conjecture was proposed:
\begin{conjecture}(Hoffmann-Ostenhof \cite{hoffman})
Let $G$ be a connected cubic graph. Then $G$ has a decomposition into a spanning tree, a matching and a family of cycles.

\end{conjecture}
Conjecture \theconjecture$\,$ also appears in Problem 516 \cite{cameron}. There are a few partial results known for Conjecture \theconjecture. Kostochka \cite{kostocha} noticed that the Petersen graph, the prisms over cycles, and many other graphs have a decomposition desired in Conjecture \theconjecture. Ozeki and Ye \cite{ozeki} proved that the conjecture holds for 3-connected cubic plane graphs. Furthermore, it was proved by Bachstein \cite{bachstein} that Conjecture \theconjecture$\,$ is true for every 3-connected cubic graph embedded in torus or Klein-bottle. Akbari, Jensen and Siggers \cite[Theorem 9]{akbari} showed that Conjecture \theconjecture$\,$ is true for Hamiltonian cubic graphs.

In this paper, we show that Conjecture \theconjecture$\,$ holds for traceable cubic graphs.
\section{Results}
Before proving the main result, we need the following lemma.
\begin{lemma}
\label{lemma:1}
Let $G$ be a cubic graph. Suppose that $V(G)$ can be partitioned into a tree $T$ and finitely many cycles such that there is no edge between any pair of cycles (not necessarily distinct cycles), and every pendant vertex of $T$ is adjacent to at least one vertex of a cycle. Then, Conjecture \theconjecture$\,$ holds for $G$.
\end{lemma}
\begin{proof}
By assumption, every vertex of each cycle in the partition is adjacent to exactly one vertex of $T$. Call the set of all edges with one endpoint in a cycle and another endpoint in $T$ by $Q$.
Clearly, the induced subgraph on $E(T) \cup Q$ is a spanning tree of $G$. We call it $T'$. Note that every edge between a pendant vertex of $T$ and the union of cycles in the partition is also contained in $T'$. Thus, every pendant vertex of $T'$ is contained in a cycle of the partition. Now, consider the graph $H = G \setminus E(T')$. For every $v \in V(T)$, $d_H(v) \leq 1$. So Conjecture \theconjecture$\,$ holds for $G$. \vspace{1em}
\end{proof}

\noindent\textbf{Remark 1.}
\label{remark:1}
Let $C$ be a cycle formed by the path $Q$. Then clearly there exists a chordless cycle formed by $Q$.

Now, we are in a position to prove the main result.

\begin{theorem}
Conjecture \theconjecture$\,$ holds for traceable cubic graphs.
\end{theorem}
\begin{proof}
Let $G$ be a traceable cubic graph and $P : v_1, \dots, v_n$ be a Hamiltonian path in $G$. By \cite[Theorem 9]{akbari}, Conjecture A holds for $v_1 v_n  \in E(G)$. Thus we can assume that $v_1 v_n  \notin E(G)$. Let $v_1 v_j, v_1 v_{j'}, v_i v_n, v_{i'} v_n \in E(G)\setminus E(P)$  and $j' < j < n$, $1 < i < i'$. Two cases can occur:
\begin{enumerate}[leftmargin=0pt,label=]
\item
\textbf{Case 1.}
Assume that $i < j$. Consider the following graph in Figure \ref{fig:overlapping} in which the thick edges denote the path $P$. Call the three paths between $v_j$ and $v_i$, from the left to the right, by $P_1$, $P_2$ and $P_3$, respectively (note that $P_1$ contains the edge $e'$ and $P_3$ contains the edge $e$).

\begin{figure}[H]
  \begin{center}
    \includegraphics[width=40mm]{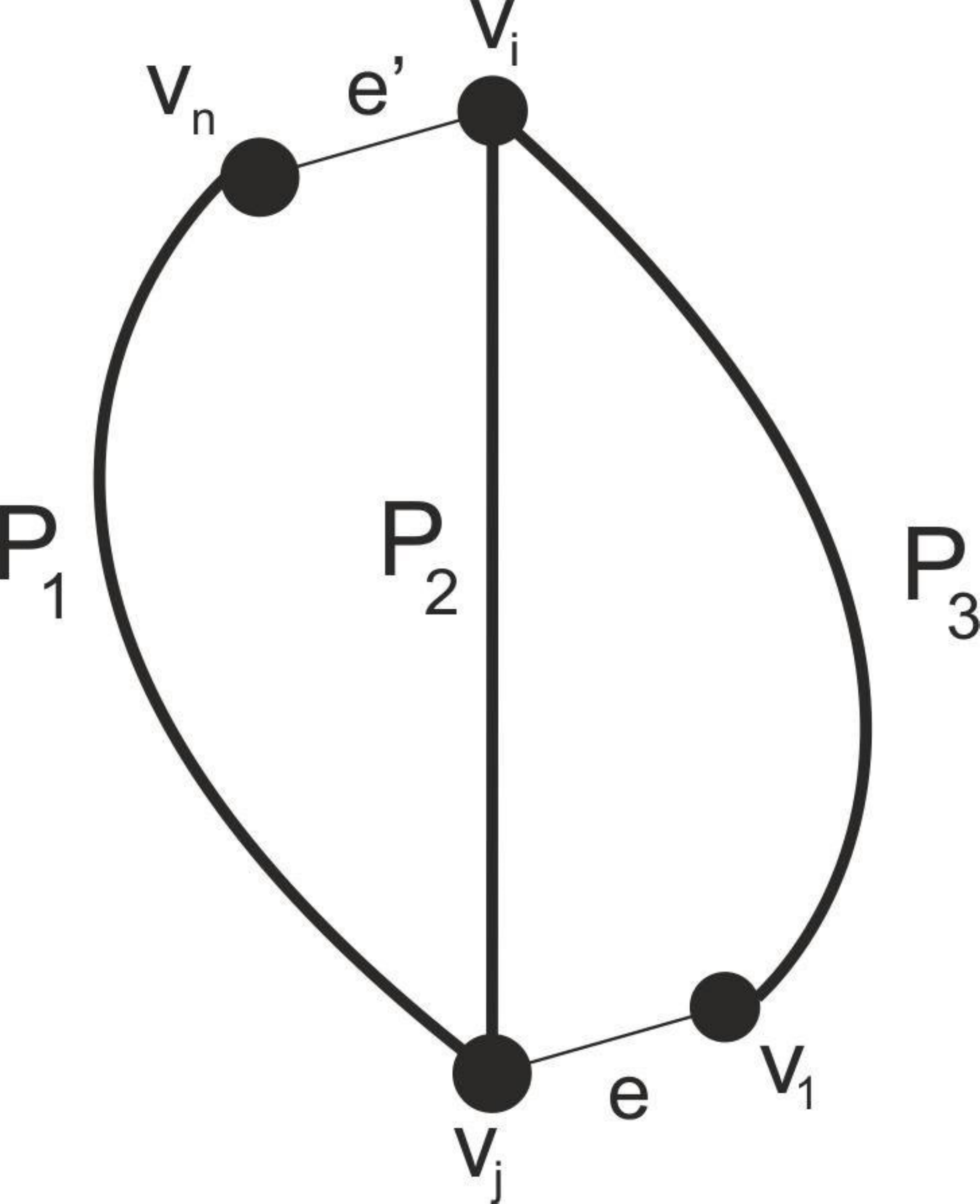}
    \caption{Paths $P_1$, $P_2$ and $P_3$}
    \label{fig:overlapping}
  \end{center}
\end{figure}

If $P_2$ has order $2$, then $G$ is Hamiltonian and so by \cite[Theorem 9]{akbari} Conjecture \theconjecture$\,$ holds. Thus we can assume that $P_1$, $P_2$ and $P_3$ have order at least $3$. Now, consider the following subcases:\\

\begin{enumerate}[leftmargin=0pt,label=]
\label{case:1}
\item \textbf{Subcase 1.} There is no edge between $V(P_r)$ and $V(P_s)$ for $1 \leq r < s \leq 3$. Since every vertex of $P_i$ has degree 3 for every $i$, by \lemmaTwo$\,$ there are two chordless cycles $C_1$ and $C_2$ formed by $P_1$ and $P_2$, respectively.
Define a tree $T$ with the edge set
$$ E\Big(\langle V(G) \setminus \big(V(C_1) \cup V(C_2)\big) \rangle\Big) \bigcap \big(\bigcup_{i=1}^3 E(P_i)\big).$$
Now, apply \lemmaOne $\,$for the partition $\{T, C_1, C_2\}$.\\

\item \textbf{Subcase 2.}
\label{case:edge}
There exists at least one edge between some $P_r$ and $P_s$, $r<s$. With no loss of generality, assume that $r=1$ and $s=2$. Suppose that $ab \in E(G)$, where $a \in V(P_1)$, $b \in V(P_2)$ and $d_{P_1}(v_j, a) + d_{P_2}(v_j, b)$ is minimum.

\begin{figure}[H]
  \begin{center}
    \includegraphics[width=40mm]{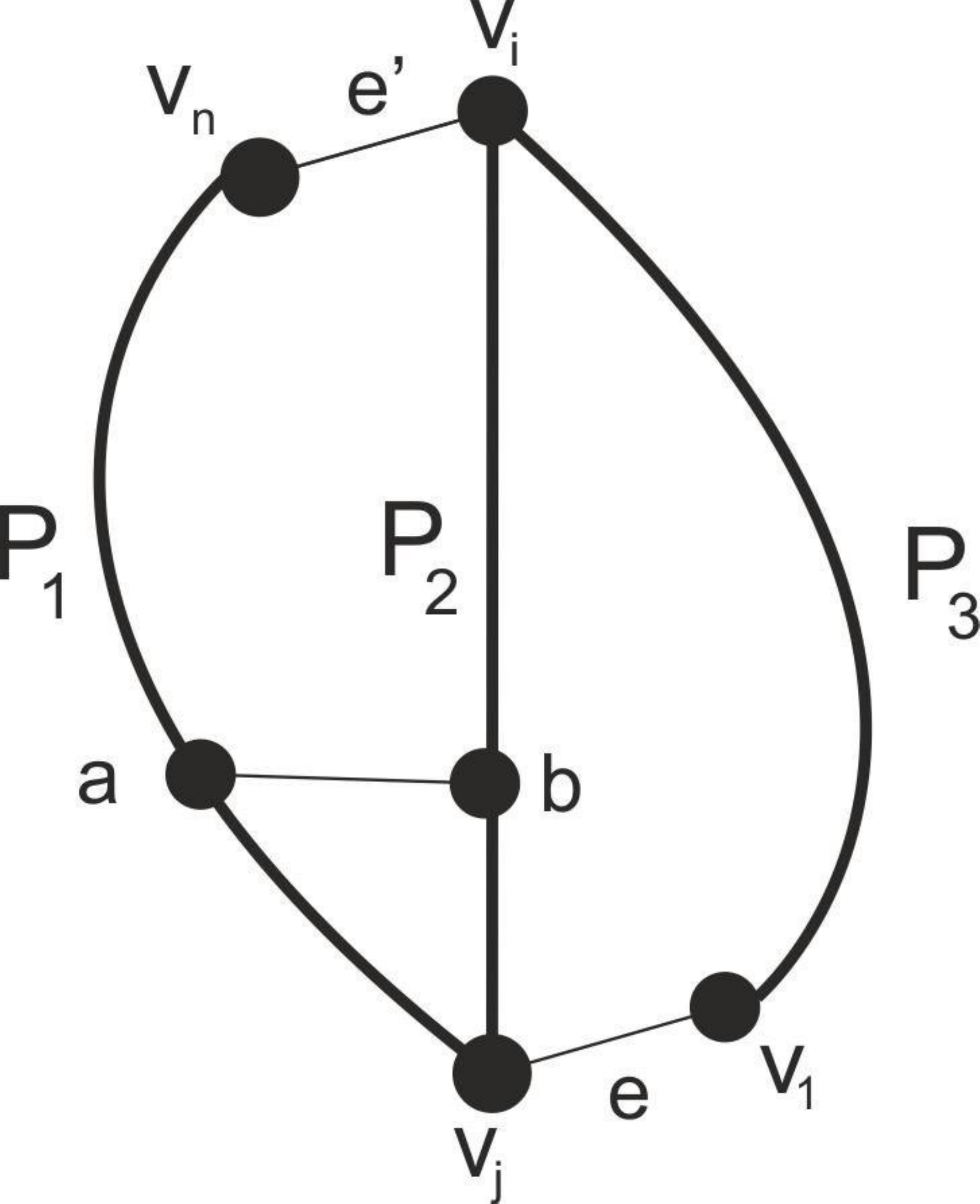}
    \caption{The edge $ab$ between $P_1$ and $P_2$}
    \label{fig:ab}
  \end{center}
\end{figure}

Three cases occur: \\

(a) There is no chordless cycle formed by either of the paths $\stackrel\frown{v_j a}$ or $\stackrel\frown{v_j b}$. Let $C$ be the chordless cycle $\stackrel\frown{v_j a}\stackrel\frown{ b v_j}$. Define $T$ with the edge set
$$ E\Big(\langle V(G) \setminus V(C)\rangle\Big) \bigcap \big(\bigcup_{i=1}^3 E(P_i)\big).$$
Now, apply \lemmaOne $\,$for the partition $\{T,C\}$.	\\

(b) There are two chordless cycles, say $C_1$ and $C_2$, respectively formed by the paths $\stackrel\frown{v_j a}$ and $\stackrel\frown{v_j b}$. Now, consider the partition $C_1$, $C_2$ and the tree induced on the following edges,
$$E\Big(\langle V(G) \setminus \big(V(C_1) \cup V(C_2)\big) \rangle\Big) \; \bigcap \; E\Big(\bigcup_{i=1}^3 P_i\Big),$$
and apply \lemmaOne.\\

(c) With no loss of generality, there exists a chordless cycle formed by the path $\stackrel\frown{v_j a}$ and there is no chordless cycle formed by the path $\stackrel\frown{v_j b}$.
First, suppose that for every chordless cycle $C_t$ on $\stackrel\frown{v_j a}$, at least one of the vertices of $C_t$ is adjacent to a vertex in $V(G) \setminus V(P_1)$.
We call one of the edges with one end in $C_t$ and other endpoint in $V(G) \setminus V(P_1)$ by $e_t$. Let $v_j=w_0, w_1, \dots, w_l=a$ be all vertices of the path $\stackrel\frown{v_j a}$ in $P_1$. Choose the shortest path $w_0 w_{i_1} w_{i_2} \dots w_l$ such that $0 < i_1 < i_2 < \dots < l$.
Define a tree $T$ whose edge set is the thin edges in Figure \ref{fig:deltaCycle}.\\
Call the cycle $w_0 w_{i_1} \dots w_l \stackrel\frown{b w_0}$ by $C'$. Now, by removing $C'$, $q$ vertex disjoint paths $Q_1, \dots, Q_q$ which are contained in $\stackrel\frown{v_j a}$ remain. Note that there exists a path of order $2$ in $C'$ which by adding this path to $Q_i$ we find a cycle $C_{t_i}$, for some $i$. Hence there exists an edge $e_{t_i}$ connecting $Q_i$ to $V(G) \setminus V(P_1)$. Now, we define a tree $T$ whose the edge set is,
$$\quad\quad\quad \bigg( E\Big(\langle V(G) \setminus V(C') \rangle \Big)\; \bigcap \; \Big(\bigcup_{i=1}^3 E(P_i)\Big) \bigg) \bigcup \Big(\big\{e_{t_i} \mid 1 \leq i \leq q \big\} \Big).$$
Apply \lemmaOne $\,$for the partition $\{T,C'\}$.\\

\begin{figure}[H]
  \begin{center}
    \includegraphics[width=40mm]{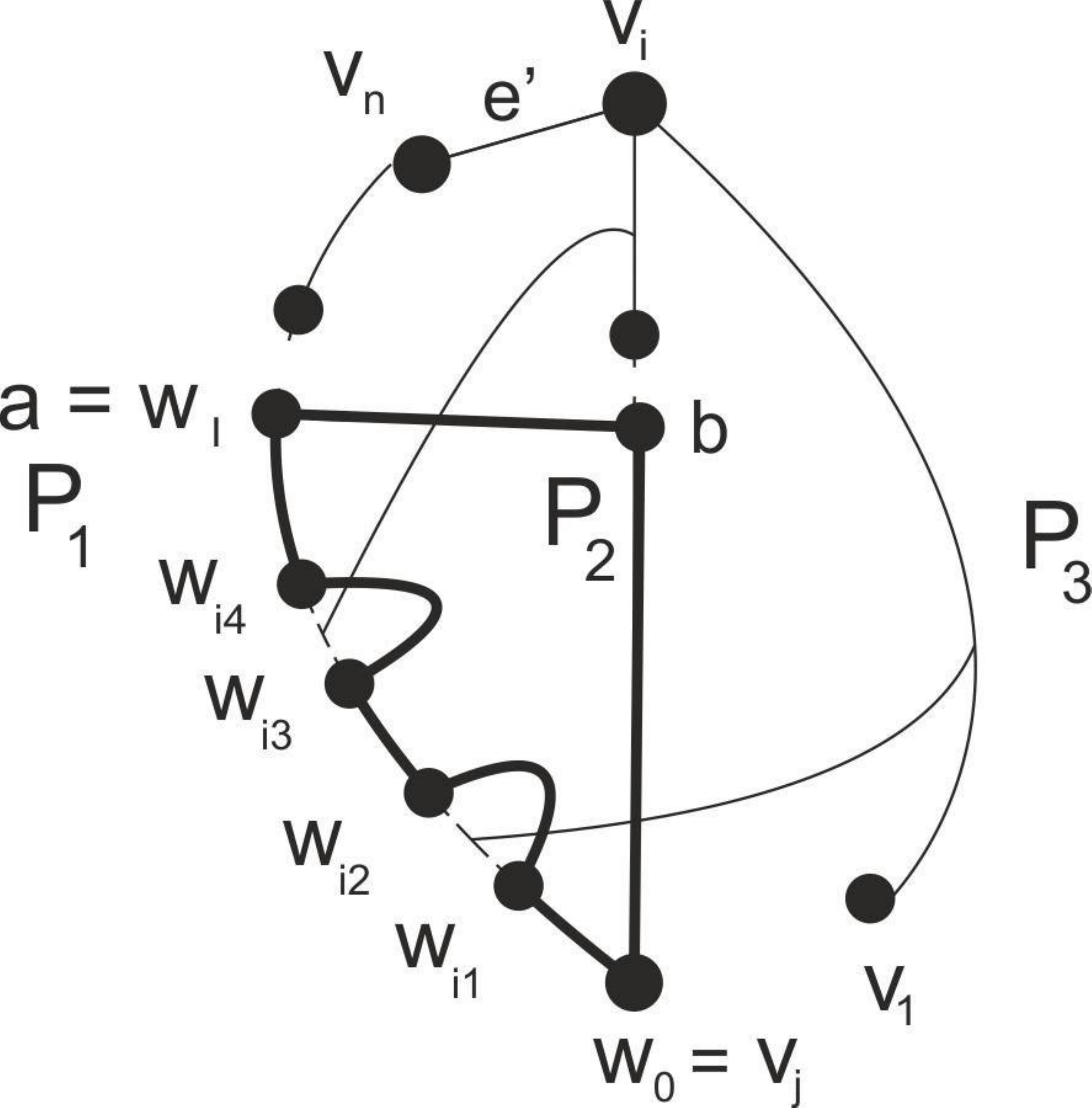}
    \caption{The cycle $C'$ and the tree $T$}
       \label{fig:deltaCycle}
  \end{center}
\end{figure}

Next, assume that there exists a cycle $C_1$ formed by $\stackrel\frown{v_j a}$ such that none of the vertices of $C_1$ is adjacent to $V(G) \setminus V(P_1)$. Choose the smallest cycle with this property. Obviously, this cycle is chordless. Now, three cases can be considered:\\

\begin{enumerate}[leftmargin=5pt,label=(\roman*)]
\item There exists a cycle $C_2$ formed by $P_2$ or $P_3$. Define the partition $C_1$, $C_2$ and a tree with the following edge set,
$$E\Big(\langle V(G) \setminus \big(V(C_1) \cup V(C_2)\big)\rangle \Big) \bigcap \Big( \bigcup_{i=1}^3 E(P_i) \Big),$$
and apply \lemmaOne.\\

\item There is no chordless cycle formed by $P_2$ and by $P_3$, and there is at least one edge between $V(P_2)$ and $V(P_3)$. Let $ab \in E(G)$, $a \in V(P_2)$ and $b \in V(P_3)$ and moreover $d_{P_2}(v_j, a) + d_{P_3}(v_j,b)$ is minimum. Notice that the cycle $\stackrel\frown{v_j a} \stackrel\frown{b v_j}$ is chordless. Let us call this cycle by $C_2$. Now, define the partition $C_2$ and a tree with the following edge set,
$$E\Big(\langle V(G) \setminus V(C_2)\rangle \Big) \bigcap \Big( \bigcup_{i=1}^3 E(P_i) \Big),$$
and apply \lemmaOne.\\

\item There is no chordless cycle formed by $P_2$ and by $P_3$, and there is no edge between $V(P_2)$ and $V(P_3)$. Let $C_2$ be the cycle consisting of two paths $P_2$ and $P_3$. Define the partition $C_2$ and a tree with the following edge set,
$$E\Big(\langle V(G) \setminus V(C_2)\rangle \Big) \bigcap \Big( \bigcup_{i=1}^3 E(P_i) \Big),$$
and apply \lemmaOne.

\end{enumerate}

\end{enumerate}

\vspace{5mm}
\item
\textbf{Case 2.}
\label{case:2}
Assume that $j < i$ for all Hamiltonian paths. Among all Hamiltonian paths consider the  path such that $i'-j'$ is maximum. Now, three cases can be considered:\\

\begin{enumerate}[leftmargin=0pt,label=]
\item \textbf{Subcase 1.} There is no $s < j'$ and $t > i'$ such that $v_s v_t \in E(G)$. By \lemmaTwo $\,$ there are two chordless cycles $C_1$ and $C_2$, respectively formed by the paths $v_1 v_{j'}$ and $v_{i'} v_n$. By assumption there is no edge $xy$, where $x \in V(C_1)$ and $y \in V(C_2)$.
Define a tree $T$ with the edge set:
$$ E\Big(\langle V(G) \setminus \big(V(C_1) \cup V(C_2)\big) \rangle \Big) \bigcap \Big( E(P) \cup \{v_{i'}v_n, v_{j'}v_1\} \Big).$$
Now, apply \lemmaOne $\,$for the partition $\{T, C_1, C_2\}$.\\

\item \textbf{Subcase 2.}
\label{subcase:22} There are at least four indices $s, s' < j$ and $t, t' > i$ such that $v_s v_t, v_{s'} v_{t'} \in E(G)$. Choose four indices $g, h < j$ and $e, f > i$ such that $v_h v_e, v_g v_f \in E(G)$ and $|g-h| + |e-f|$ is minimum.

\begin{figure}[H]
  \begin{center}
    \includegraphics[width=90mm]{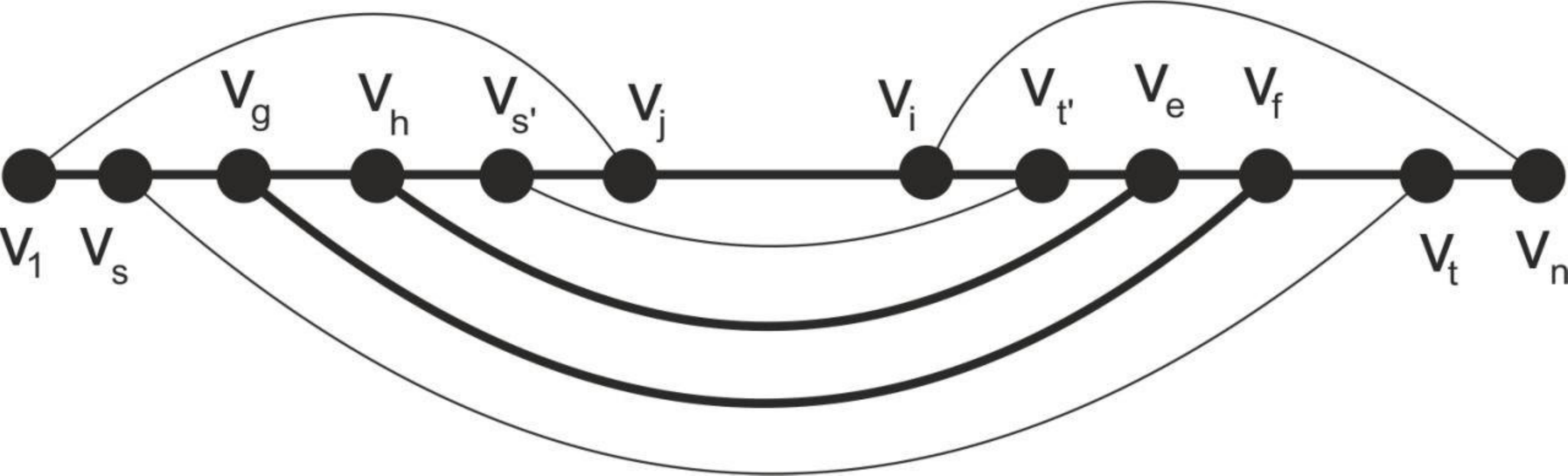}
    \caption{Two edges $v_h v_e$ and $v_g v_f$}
    \label{fig:non-overlapping}
  \end{center}
\end{figure}

Three cases can be considered:\\

\begin{enumerate}[leftmargin=0pt,label=(\alph*)]
\item There is no chordless cycle formed by $\stackrel\frown{v_g v_h}$ and by $\stackrel\frown{v_e v_f}$.

Consider the cycle $\stackrel\frown{v_g v_h} \stackrel\frown{v_e v_f}v_g$ and call it $C$. Now, define a tree $T$ with the edge set,
$$\,\,\,E\Big(\langle V(G) \setminus V(C)\rangle \Big) \bigcap \Big( E(P) \cup \{v_1v_{j}, v_{i}v_n\} \Big),$$
apply \lemmaOne $\,$for the partition $\{T, C\}$.\\

\item With no loss of generality, there exists a chordless cycle formed by $\stackrel\frown{v_e v_f}$ and there is no chordless cycle formed by the path $\stackrel\frown{v_g v_h}$. First suppose that there is a chordless cycle $C_1$ formed by $\stackrel\frown{v_e v_f}$ such that there is no edge between $V(C_1)$ and $\{v_1, \dots, v_j\}$. By \lemmaTwo $,$ there exists a chordless cycle $C_2$ formed by $\stackrel\frown{v_1 v_j}$. By assumption there is no edge between $V(C_1)$ and $V(C_2)$. Now, define a tree $T$ with the edge set,

$$\quad\quad\quad\quad E\Big(\langle V(G) \setminus \big(V(C_1) \cup V(C_2)\big)\rangle \Big) \bigcap \Big( E(P) \cup \{v_1v_{j}, v_{i}v_n\} \Big),$$

and apply \lemmaOne $\,$for the partition $\{T, C_1, C_2\}$.

$\;$ Next assume that for every cycle $C_r$ formed by $\stackrel\frown{v_e v_f}$, there are two vertices $x_r \in V(C_r)$ and $y_r \in \{v_1, \dots, v_j\}$ such that $x_r y_r \in E(G)$. Let $v_e=w_0, w_1, \dots, w_l=v_f$ be all vertices of the path $\stackrel\frown{v_e v_f}$ in $P$. Choose the shortest path $w_0 w_{i_1} w_{i_2} \dots w_l$ such that $0 < i_1 < i_2 < \dots < l$. Consider the cycle $w_0 w_{i_1} \dots w_l \stackrel\frown{v_g v_h}$ and call it $C$. Now, by removing $C$, $q$ vertex disjoint paths $Q_1, \dots, Q_q$ which are contained in $\stackrel\frown{v_e v_f}$ remain. Note that there exists a path of order $2$ in $C$ which by adding this path to $Q_i$ we find a cycle $C_{r_i}$, for some $i$. Hence there exists an edge $x_{r_i} y_{r_i}$ connecting $Q_i$ to $V(G) \setminus V(\stackrel\frown{v_e v_f})$. We define a tree $T$ whose edge set is the edges,
$$\quad\quad\quad\quad\quad\quad E\Big(\langle V(G) \setminus V(C)\rangle \Big) \bigcap \Big( E(P) \cup \{v_1v_{j}, v_{i}v_n\} \cup \big\{x_{r_i} y_{r_i} \mid 1 \leq i \leq q\big\} \Big),$$
then apply \lemmaOne $\,$ on the partition $\{T, C\}$.\\
\begin{figure}[H]
  \begin{center}
    \includegraphics[width=90mm]{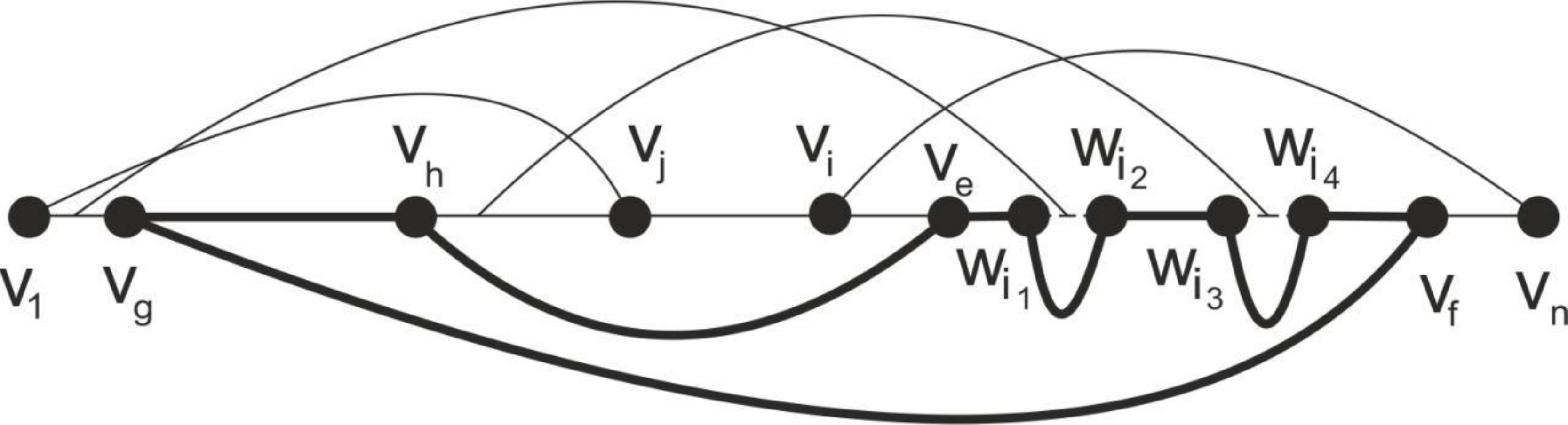}
    \caption{The tree $T$ and the shortest path $w_0 w_{i_1}\dots w_l$}
    \label{fig:delta-non-overlapping}
  \end{center}
\end{figure}

\item There are at least two chordless cycles, say $C_1$ and $C_2$ formed by the paths $\stackrel\frown{v_g v_h}$ and $\stackrel\frown{v_e v_f}$, respectively. Since $|g-h| + |e-f|$ is minimum, there is no edge $xy \in E(G)$ with $x \in V(C_1)$ and $y \in V(C_2)$. Now, define a tree $T$ with the edge set,
$$\quad\quad\quad\quad E\Big( \langle V(G) \setminus \big(V(C_1) \cup V(C_2)\big) \rangle \Big) \bigcap \Big( E(P) \cup \{v_1 v_{j}, v_{i}v_n\} \Big),$$
and apply \lemmaOne $\,$for the partition $\{T, C_1, C_2\}$.\\
\end{enumerate}

\item \textbf{Subcase 3.} There exist exactly two indices $s,t$,  $s < j' < i' < t$ such that $v_s v_t \in E(G)$ and there are no two other indices $s', t'$ such that $s' < j < i < t'$ and $v_{s'} v_{t'} \in E(G)$. We can assume that there is no cycle formed by  $\stackrel\frown{v_{s+1} v_j}$ or $\stackrel\frown{v_i v_{t-1}}$, to see this by symmetry consider a cycle $C$ formed by $\stackrel\frown{v_{s+1} v_j}$. By \lemmaTwo $\,$ there exist chordless cycles $C_1$ formed by $\stackrel\frown{v_{s+1} v_j}$ and $C_2$ formed by $\stackrel\frown{v_{i} v_n}$. By assumption $v_s v_t$ is the only edge such that $s < j$ and $t > i \;$. Therefore,  there is no edge between $V(C_1)$ and  $V(C_2)$. Now, let $T$ be a tree defined by the edge set,
$$ E\Big(\langle V(G) \setminus \big(V(C_1) \cup V(C_2)\big)\rangle \Big) \bigcap \Big( E(P) \cup \{v_1v_{j}, v_{i}v_n\} \Big),$$
and apply \lemmaOne $\,$for the partition \{$T$, $C_1$, $C_2$\}.\\

$\quad$Furthermore, we can also assume that either $s \neq j'-1$  or $t \neq i'+1$, otherwise we have the Hamiltonian cycle $\stackrel\frown{v_1 v_s} \stackrel\frown{v_t v_n} \stackrel\frown{v_{i'} v_{j'}} v_1$ and by \cite[Theorem 9]{akbari} Conjecture \theconjecture$\,$ holds.

$\quad$By symmetry, suppose that $s \neq j'-1$. Let $v_k$ be the vertex adjacent to $v_{j'-1}$, and $k \notin \{j'-2, j'\}$. It can be shown that $k > j'-1$, since otherwise by considering the Hamiltonian path $P': \; \stackrel\frown{ v_{k+1} v_{j'-1}}\stackrel\frown{v_k v_1} \stackrel\frown{v_{j'} v_n}$,  the new $i'-j'$ is greater than the old one and this contradicts our assumption about $P$ in the \hyperref[case:2]{Case 2}.

$\quad$We know that $j' < k < i$. Moreover, the fact that  $\stackrel\frown{v_{s+1} v_j}$ does not form a cycle contradicts the case that $j' < k \le j$. So $j < k < i$. Consider two cycles $C_1$ and $C_2$, respectively with the vertices $v_1 \stackrel\frown{v_{j'} v_{j}} v_1$ and $v_n \stackrel\frown{v_{i'} v_{i}} v_n$. The cycles $C_1$ and $C_2$ are chordless, otherwise there exist cycles formed by the paths $\stackrel\frown{v_{s+1} v_j}$ or $\stackrel\frown{v_i v_{t-1}}$. Now, define a tree $T$ with the edge set
$$ E\Big(\langle V(G) \setminus \big(V(C_1) \cup V(C_2)\big)\rangle \Big) \bigcap \Big( E(P) \cup \{v_s v_t, v_k v_{j'-1}\} \Big),$$
and apply \lemmaOne $\,$for the partition \{$T$, $C_1$, $C_2$\}.
\end{enumerate}
\end{enumerate}
\end{proof}

\noindent\textbf{Remark 2.}
\label{remark:2}
Indeed, in the proof of the previous theorem we showed a stronger result, that is, for every traceable cubic graph there is a decomposition with at most two cycles.

\end{document}